\documentclass[a4paper,11pt,reqno]{amsart}
\usepackage{amsmath}
\usepackage{amstext}
\usepackage{amsbsy}
\usepackage{amsopn}
\usepackage{upref}
\usepackage{amsthm}
\usepackage{amsfonts}
\usepackage{amssymb}
\usepackage{mathrsfs}
\usepackage{times}
\usepackage{cite}
\allowdisplaybreaks
\numberwithin{equation}{section}
 \newtheorem{theorem}{Theorem}

 \newtheorem{lemma}[theorem]{Lemma}
\theoremstyle{definition}
 \newtheorem{definition}[theorem]{Definition}
\theoremstyle{remark}
 
\begin{document}
\title[Radon transform]{Inversion of seismic-type Radon transforms on the plane}
\author[H.~Chihara]{Hiroyuki Chihara}
\address{College of Education, University of the Ryukyus, Nishihara, Okinawa 903-0213, Japan}
\email{aji@rencho.me}
\begin{abstract}
We study integral transforms mapping a function on the Euclidean plane 
to the family of its integration on plane curves, 
that is, a function of plane curves. 
The plane curves we consider in the present paper are given 
by the graphs of functions with a fixed axis of the independent variable, 
and are imposed some symmetry with respect to the axes.  
These transforms contain 
the parabolic Radon transform 
and the hyperbolic Radon transform 
arising from seismology. 
We prove the inversion formulas for these transforms 
under some vanishing and symmetry conditions of functions. 
\end{abstract}
\maketitle
\section{Introduction}
\label{section:introduction}
Fix arbitrary $c\in\mathbb{R}$ and $\alpha,\beta>1$. 
Let $(x,y), (s,u) \in \mathbb{R}^2$ be variables of functions. 
We study three types of integral transforms 
$\mathcal{P}_\alpha f(s,u)$, 
$\mathcal{Q}_\alpha f(s,u)$
and 
$\mathcal{R}_{\alpha,\beta} f(s,u)$ 
of a function $f(x,y)$. 
These are the integration of $f(x,y)$ on the graph of functions 
of the fixed axis $x$. 
Their precise definitions are the following. 
\par
Firstly, 
$\mathcal{P}_\alpha f(s,u)$ is defined by 
$$
\mathcal{P}_{\alpha}f(s,u)
=
\int_{-\infty}^\infty
f(x,s\lvert{x-c}\rvert^\alpha+u)
dx
=
\int_{-\infty}^\infty
f(x+c,s\lvert{x}\rvert^\alpha+u)
dx. 
$$
$\mathcal{P}_{\alpha}f(s,u)$ is the integration of $f$ 
on a curve 
$$
\Gamma_\mathcal{P}(\alpha;s,u)
=
\{(x,s\lvert{x-c}\rvert^\alpha+u)\ : \ x\in\mathbb{R}\}, 
$$
and slightly different from the standard contour integral on 
$\Gamma_\mathcal{P}(\alpha;s,u)$ because the integration does not contain 
the line element. 
In particular $\Gamma_\mathcal{P}(2l;s,u)$ with $l=1,2,3,\dotsc$ and $s\ne0$ 
is a translation and a vertical dilation of the graph of monomial $x^{2l}$.  
Moreover $\mathcal{P}_2f(s,u)$ is called 
the parabolic Radon transform of $f$ in seismology.  
Note that 
$$
f(-x+c,s\lvert{-x}\rvert^\alpha+u)=f(-x+c,s\lvert{x}\rvert^\alpha+u)
$$
and $f$ can be splitted into 
$$
f(x+c,y)
=
\frac{f(x+c,y)+f(-x+c,y)}{2}
+
\frac{f(x+c,y)-f(-x+c,y)}{2}.
$$
The second term of the right hand side of the above 
does not contribute to $\mathcal{P}_{\alpha}f(s,u)$. 
So it is natural to consider functions satisfying 
$f(-x+c,y)=f(x+c,y)$ for $\mathcal{P}_\alpha$. 
Otherwise the injectivity of $\mathcal{P}_\alpha$ breaks down, 
and the inversion formula for $\mathcal{P}_\alpha$ cannot be expected. 
\par
Secondly, 
$\mathcal{Q}_\alpha f(s,u)$ is defined by 
$$
\mathcal{Q}_{\alpha}f(s,u)
=
\int_{-\infty}^\infty
f(x,s(x-c)\lvert{x-c}\rvert^{\alpha-1}+u)
dx
=
\int_{-\infty}^\infty
f(x+c,sx\lvert{x}\rvert^{\alpha-1}+u)
dx. 
$$
$\mathcal{Q}_{\alpha}f(s,u)$ is the integration of $f$ 
on a curve 
$$
\Gamma_\mathcal{Q}(\alpha;s,u)
=
\{(x,s(x-c)\lvert{x-c}\rvert^{\alpha-1}+u)\ : \ x\in\mathbb{R}\}, 
$$
and slightly different from the standard contour integral on 
$\Gamma_Q(\alpha;s,u)$. 
In particular $\Gamma_Q(2l+1;s,u)$ with $l=1,2,3,\dotsc$ and $s\ne0$ 
is a translation and a vertical dilation of the graph of $x^{2l+1}$. 
Unfortunately, however, $\mathcal{Q}_{\alpha}f(s,u)$ 
does not seem to be used in science and technology. 
\par
Finally, 
$\mathcal{R}_{\alpha,\beta}f(s,u)$ is defined by 
\begin{align*}
  \mathcal{R}_{\alpha,\beta}f(s,u)
& =
  \int_{\substack{x\in\mathbb{R} \\ s\lvert{x-c}\rvert^\alpha+u>0}}
  \frac{f\bigl(x,(s\lvert{x-c}\rvert^\alpha+u)^{1/\beta}\bigr)}{(s\lvert{x-c}\rvert^\alpha+u)^{1/\beta}}
  dx
\\
& =
  \int_{\substack{x\in\mathbb{R} \\ s\lvert{x}\rvert^\alpha+u>0}}
  \frac{f\bigl(x+c,(s\lvert{x}\rvert^\alpha+u)^{1/\beta}\bigr)}{(s\lvert{x}\rvert^\alpha+u)^{1/\beta}}
  dx. 
\end{align*}
It is natural to impose $f(x,0)=0$ on $f(x,y)$ 
in order to resolve the singularities at $s\lvert{x}\rvert^\alpha+u=0$. 
Note that 
$$
\frac{f\bigl(-x+c,(s\lvert{-x}\rvert^\alpha+u)^{1/\beta}\bigr)}{(s\lvert{-x}\rvert^\alpha+u)^{1/\beta}}
=
\frac{f\bigl(-x+c,(s\lvert{x}\rvert^\alpha+u)^{1/\beta}\bigr)}{(s\lvert{x}\rvert^\alpha+u)^{1/\beta}}. 
$$
Moreover, if $f(x,y)$ satisfies the symmetry 
$$
f(x+c,y)=f(-x+c,y)=f(x+c,-y), 
$$
$\mathcal{R}_{\alpha,\beta}f(s,u)$ 
can be regarded as the integration of $f$ on a curve 
$$
\Gamma_\mathcal{R}(\alpha,\beta;s,u)
=
\{
(x,y)\in\mathbb{R}^2
\ : \ 
\lvert{y}\rvert^\beta=s\lvert{x-c}\rvert^\alpha+u
\}, 
$$
and slightly different from the standard contour integral on 
$\Gamma_\mathcal{R}(\alpha,\beta;s,u)$. 
In particular $\Gamma_\mathcal{R}(2,2;s,u)$ with $s>0$  
is a hyperbola, and 
$\mathcal{R}_{2,2}f(s,u)$ 
is called the hyperbolic Radon transform of $f$ in seismology. 
\par
Here we recall the mathematical background of our transforms. 
In the early 1980s, 
Cormack introduced the Radon transform of a family of plane curves and 
studied the basic properties in his pioneering works 
\cite{Cormack1981} and \cite{Cormack1982}. 
In 1998 Denecker, van Overloop and Sommen 
in \cite{Denecker1998} studied the parabolic Radon transform 
without fixed axis, in particular, the support theorem, 
higher dimensional generalization and etc. 
More than a decade later, Jollivet, Nguyen and Truong in \cite{Jollivet2011} 
studied some properties of the parabolic Radon transform with fixed axis, 
which is the exact contour integration. 
Recently, Moon established the inversion of 
the parabolic Radon transform $\mathcal{P}_2$ 
and the inversion of the hyperbolic Radon transform 
$\mathcal{R}_{2,2}$ respectively in his interesting paper \cite{Moon2016}. 
He introduced some change of variables in $(x,y)\in\mathbb{R}^2$ 
so that the Radon transform over a family of plane curves became 
the X-ray transform, that is, the Radon transform over a family of lines. 
Functions considered in \cite{Moon2016} 
are supposed to be compactly supported in 
$\mathbb{R}\setminus\{c\} \times \mathbb{R}$. 
More recently, replacing $x^2$ by some function $\varphi(x)$ 
in the parabolic Radon transform, 
Ustaoglu developed Moon's idea 
to study the inversion of more general Radon transforms 
on the plane in \cite{Ustaoglu2017}. 
He considered the inversion formulas for functions 
which are compactly supported in $\mathbb{R}\setminus\{c\} \times \mathbb{R}$. 
This means that such a function is identically equal to zero near $x=c$, 
and this is a strong vanishing condition at $x=c$. 
Moreover, we remark that the inversion formulas in 
\cite{Moon2016} and \cite{Ustaoglu2017} 
are proved by establishing the Fourier slice theorem and by using the Fourier inversion formula.   
\par
Now we also recall the scientific background of 
the parabolic Radon transform $\mathcal{P}_2$ 
and the hyperbolic transform $\mathcal{R}_{2,2}$. 
These are used for processing of 
the data of observation of seismic waves. 
In our notation $x$ means the distance of the source and 
$y$ means the travel time of the wave. 
In an early stage, 
the hyperbolic Radon transform seemed to have been used in seismology. 
In 1986, Hampton proposed to replace hyperbolas by parabolas 
and introduced the parabolic Radon transform in \cite{Hampson1986}. 
See also expository papers due to seismologists 
\cite{Maeland1998} for the parabolic Radon transform and 
\cite{Bickel2000} for the hyperbolic Radon transform respectively.   
\par
The purpose of the present paper is 
to establish the inversion formulas for 
$\mathcal{P}_\alpha$, $\mathcal{Q}_\alpha$ and $\mathcal{R}_{\alpha,\beta}$, 
which are the generalization of 
the parabolic Radon transform and the hyperbolic Radon transform 
with some symmetry of curves. 
In particular, 
we shall prove the inversion formulas elementarily 
under weaker conditions on the function $f(x,y)$. 
More precisely, we impose a finite order vanishing condition on $f(x,y)$ at $x=c$.  
Here we introduce some function spaces to state our results. 
These function spaces consists of Schwartz functions on $\mathbb{R}^2$ 
satisfying some symmetries and vanishing conditions. 
The symmetries are natural for our transforms, 
and the vanishing conditions are used 
for justifying the change of variables 
for the reduction to the X-ray transform. 
For the sake of simplicity, all the definitions, 
lemmas and theorems are stated in the order for 
$\mathcal{Q}_\alpha$, $\mathcal{P}_\alpha$ and $\mathcal{R}_{\alpha,\beta}$. 
We denote the set of all Schwartz functions on $\mathbb{R}^2$ by 
$\mathscr{S}(\mathbb{R}^2)$, that is, we say that $f(x,y) \in \mathscr{S}(\mathbb{R}^2)$ if 
$f(x,y) \in C^\infty(\mathbb{R}^2)$ and 
\begin{equation}
\sup_{(x,y)\in\mathbb{R}^2}
\left\{
(1+\lvert{x}\rvert+\lvert{y}\rvert)^N
\left\lvert
\frac{\partial^{k+l}f}{\partial x^k \partial y^l}
(x,y)
\right\rvert
\right\}
<\infty
\label{equation:schwartz} 
\end{equation}  
for any $k,l,N=0,1,2,\dotsc$. 
\begin{definition}
\label{definition:shwartzQ} 
Fix $c\in\mathbb{R}$. Let $m$ be a nonnegative integer. 
We define functions spaces 
$\mathscr{S}_{c,m}(\mathbb{R}^2)$, 
$\mathscr{S}_{c,m}^\mathcal{P}(\mathbb{R}^2)$ 
and 
$\mathscr{S}_{c,m}^\mathcal{R}(\mathbb{R}^2)$ 
by 
\begin{align*}
  \mathscr{S}_{c,m}(\mathbb{R}^2)
& =
  \left\{
  f\in\mathscr{S}(\mathbb{R}^2)
  \ : \ 
  \frac{\partial^k f}{\partial x^k}(c,y)=0
  \quad\text{for}\quad
  k=0,1,\dotsc,m
  \right\}, 
\\
  \mathscr{S}_{c,m}^\mathcal{P}(\mathbb{R}^2)
& =
  \{
  f\in\mathscr{S}_{c,m}(\mathbb{R}^2)
  \ : \ 
  f(x+c,y)=f(-x+c,y)\},
\\ 
  \mathscr{S}_{c,m}^\mathcal{R}(\mathbb{R}^2)
& =
  \{
  f\in\mathscr{S}_{c,m}^\mathcal{P}(\mathbb{R}^2)
  \ : \ 
  f(x+c,y)=f(x+c,-y),\ f(x+c,0)=0\}.
\end{align*}
\end{definition}
Recall $\alpha>1$ and $\beta>1$. 
Throughout the present paper, 
we here assume that the vanishing order $m$ at $x=c$ 
satisfies $m \geqq \alpha-2$.  
This condition guarantees the existence of finite boundary value 
at $x \rightarrow c\pm0$ after the reduction to X-ray transform. 
Our main results are the following. 
\begin{theorem}
\label{theorem:main}
Let $c\in\mathbb{R}$ and let $\alpha,\beta>1$. 
Suppose that $m$ is a nonnegative integer satisfying $m\geqq\alpha-2$. 
\begin{itemize}
\item[{\rm (i)}] 
For any $f\in\mathscr{S}_{c,m}(\mathbb{R}^2)$, 
\begin{equation}
f(x,y)
=
\frac{\alpha\lvert{x-c}\rvert^{\alpha-1}}{2\pi^2}
\int_{-\infty}^\infty
\left(
\operatorname{vp}
\int_{-\infty}^\infty
\frac{\partial_u \mathcal{Q}_{\alpha}f(s,u)}{y-s(x-c)\lvert{x-c}\rvert^{\alpha-1}-u}
du
\right)
ds.
\label{equation:inversion1} 
\end{equation}
\item[{\rm (ii)}] 
For any $f\in\mathscr{S}_{c,m}^\mathcal{P}(\mathbb{R}^2)$, 
\begin{equation}
f(x,y)
=
\frac{\alpha\lvert{x-c}\rvert^{\alpha-1}}{4\pi^2}
\int_{-\infty}^\infty
\left(
\operatorname{vp}
\int_{-\infty}^\infty
\frac{\partial_u \mathcal{P}_{\alpha}f(s,u)}{y-s\lvert{x-c}\rvert^\alpha-u}
du
\right)
ds.
\label{equation:inversion2} 
\end{equation}
\item[{\rm (iii)}]  
For any $f\in\mathscr{S}_{c,m}^\mathcal{R}(\mathbb{R}^2)$, 
\begin{equation}
f(x,y)
=
\frac{\alpha\lvert{x-c}\rvert^{\alpha-1}\lvert{y}\rvert}{4\pi^2}
\int_{-\infty}^\infty
\left(
\operatorname{vp}
\int_{-\infty}^\infty
\frac{\partial_u \mathcal{R}_{\alpha,\beta}f(s,u)}{\lvert{y}\rvert^\beta-s\lvert{x-c}\rvert^\alpha-u}
du
\right)
ds.
\label{equation:inversion3} 
\end{equation}
\end{itemize}
Here $\partial_u=\partial/\partial u$ and 
$\operatorname{vp}$ denotes the Cauchy principal value for improper integrals.  
\end{theorem}
Note that \eqref{equation:inversion1} and \eqref{equation:inversion2} were proved 
in \cite{Ustaoglu2017} 
for compactly supported smooth functions in $\mathbb{R}\setminus\{c\}\times\mathbb{R}$. 
Moreover, note that 
\eqref{equation:inversion2} for $\alpha=2$ 
and \eqref{equation:inversion3} for $\alpha=\beta=2$ 
were obtained in \cite{Moon2016}, 
but the vanishing conditions on $f(x,y)$ are quite strong. 
\par
The proof of Theorem~\ref{theorem:main} depends on 
Moon's idea of the reduction to X-ray transform in \cite{Moon2016}. 
We apply some change of variables to the inversion formula of the X-ray transform, 
we can obtain our inversion formulas. The author believes that this idea makes the proof simple 
in comparison with the proof via Fourier slice theorem for these transforms.   
We prepare some lemmas in Section~\ref{section:preliminaries}, 
and prove Theorem~\ref{theorem:main} in Section~\ref{section:proofs}.  
\section{Preliminaries}
\label{section:preliminaries}
We begin with the X-ray transform on $\mathbb{R}^2$. 
\begin{definition}
\label{definition:X-ray}
The X-ray transform of $f\in\mathscr{S}(\mathbb{R}^2)$ is defined by 
$$
\mathcal{X}f(\theta,t)
=
\int_{-\infty}^\infty
f(t\cos\theta-\sigma\sin\theta,t\sin\theta+\sigma\cos\theta)
d\sigma
$$
for $(\theta,t)\in\mathbb{R}^2$. 
Note that $\mathcal{X}f(\theta\pm\pi,-t)=\mathcal{X}f(\theta,t)$.  
\end{definition}
The inversion formula of the X-ray transform is well-known as follows.  
\begin{theorem}
\label{theorem:inversionX}
For $f\in\mathscr{S}(\mathbb{R}^2)$, 
$$
f(x,y)
=
\frac{1}{2\pi^2}
\int_0^\pi
\left(
\operatorname{vp}
\int_{-\infty}^\infty
\frac{\partial_t \mathcal{X}f(\theta,t)}{x\cos\theta+y\sin\theta-t}
dt
\right)
d\theta.
$$
\end{theorem}
It is very important in the present paper that 
the X-ray transform and its inversion formula are valid also for 
a smooth function $f(x,y)$ satisfying 
$f(x,y)=O\bigl((1+\lvert{x}\rvert+\lvert{y}\rvert)^{-d}\bigr)$ 
with some $d>1$, 
compactly supported distributions, 
rapidly decaying Lebesgue measurable functions and etc.  
See, e.g., Chapter~1 of Helgason's textbook \cite{Helgason2011} for the detail. 
\par
Secondly, we give a lemma to make full use of the vanishing conditions. 
\begin{lemma}
\label{theorem:vanishing}
\quad
\begin{itemize}
\item[{\rm (i)}] 
For $f\in\mathscr{S}_{c,m}(\mathbb{R}^2)$, 
\begin{equation}
f(x+c,y)
=
\frac{x^{m+1}}{m!}
\int_0^1
(1-t)^m
\frac{\partial^{m+1} f}{\partial x^{m+1}}(tx+c,y)
dt.
\label{equation:vanishing1} 
\end{equation}
\item[{\rm (ii)}] 
For $f\in\mathscr{S}_{c,m}^\mathcal{R}(\mathbb{R}^2)$, 
\begin{equation}
\frac{\partial^l f}{\partial x^l}(x+c,y)
=
y
\int_0^1
\frac{\partial^{l+1} f}{\partial x^l \partial y}(x+c,\tau y)
d\tau,
\quad
l=0,1,2,\dotsc,
\label{equation:vanishing2} 
\end{equation}
\begin{equation}
f(x+c,y)
=
\frac{x^{m+1}y}{m!}
\int_0^1
\int_0^1
(1-t)^m
\frac{\partial^{m+2} f}{\partial x^{m+1} \partial y}(tx+c,\tau y)
dtd\tau.
\label{equation:vanishing3} 
\end{equation}

\end{itemize}
\end{lemma}
\begin{proof}
Let $f\in\mathscr{S}_{c,m}(\mathbb{R}^2)$. 
Then $\partial^kf/\partial x^k(c,y)=0$ for $k=0,1,\dotsc,m$. 
Taylor's formula in $x$ gives 
\begin{align*}
  f(x+c,y)
& =
  \sum_{k=0}^m
  \frac{x^k}{k!}
  \frac{\partial^k f}{\partial x^k}(c,y)
  +
  \frac{x^{m+1}}{m!}
  \int_0^1
  (1-t)^m
  \frac{\partial^{m+1} f}{\partial x^{m+1}}
  (tx+c,y)
  dt
\\
& =
  \frac{x^{m+1}}{m!}
  \int_0^1
  (1-t)^m
  \frac{\partial^{m+1} f}{\partial x^{m+1}}
  (tx+c,y)
  dt,
\end{align*}
which is \eqref{equation:vanishing1}. 
\par
Let $f\in\mathscr{S}_{c,m}^\mathcal{R}(\mathbb{R}^2)$. 
Then \eqref{equation:vanishing1} holds since 
$\mathscr{S}_{c,m}^\mathcal{R}(\mathbb{R}^2)\subset\mathscr{S}_{c,m}(\mathbb{R}^2)$. 
Recall $f(x+c,0)=0$. 
Then $\partial^l f/\partial x^l(x+c,0)=0$ for $l=0,1,2,\dotsc$. 
Applying the mean value theorem to 
$\partial^l f/\partial x^l(x+c,y)$ in $y$, 
we have 
$$
\frac{\partial^l f}{\partial x^l}(x+c,y)
=
\frac{\partial^l f}{\partial x^l}(x+c,y)
-
\frac{\partial^l f}{\partial x^l}(x+c,0)
=
y
\int_0^1
\frac{\partial^{l+1} f}{\partial x^l \partial y}(x+c,\tau y)
d\tau, 
$$
which is \eqref{equation:vanishing2}. 
Applying this with $l={m+1}$ to \eqref{equation:vanishing1}, 
we obtain \eqref{equation:vanishing3}. 
This completes the proof.  
\end{proof}
We make use of the change of variable 
$x=\xi\lvert\xi\rvert^{-1+1/\alpha}=\operatorname{sgn}(\xi)\lvert\xi\rvert^{1/\alpha}$ for $\xi\ne0$ in order to reduce our transforms to the X-ray transform. 
Since $\lvert{x}\rvert=\lvert\xi\rvert^{1/\alpha}$, 
we have $sx\lvert{x}\rvert^{\alpha-1}+u=s\xi+u$ for $\xi\ne0$ 
and in particular $s\lvert{x}\rvert^\alpha+u=s\xi+u$ for $\xi>0$. 
For this purpose, we introduce new functions of 
$(\xi,\eta)\in\mathbb{R}^2$ defined by $f(x,y)$ as follows: 
\begin{align*}
  F_\alpha(\xi,\eta)
& =
  \dfrac{f(\xi\lvert\xi\rvert^{-1+1/\alpha}+c,\eta)}{\alpha\lvert\xi\rvert^{(\alpha-1)/\alpha}}
  \qquad
  (\xi\ne0),
\\
  F_\alpha^\mathcal{P}(\xi,\eta)
& =
  \begin{cases}
  2F_\alpha(\xi,\eta) &\quad (\xi>0),
  \\
  0 &\quad (\text{otherwise}),   
  \end{cases}
  =
  \begin{cases}
  \dfrac{2f(\xi^{1/\alpha}+c,\eta)}{\alpha\xi^{(\alpha-1)/\alpha}}
  & \quad (\xi>0), 
  \\
  0 &\quad (\text{otherwise}),     
  \end{cases}
\\
  F_{\alpha,\beta}^\mathcal{R}(\xi,\eta)
& =  
  \begin{cases}   
  \dfrac{2f(\xi^{1/\alpha}+c,\eta^{1/\beta})}{\alpha\xi^{(\alpha-1)/\alpha}\eta^{1/\beta}}
  & \quad (\xi>0, \eta>0), 
  \\
  0 &\quad (\text{otherwise}).    
  \end{cases}
\end{align*}
We consider the properties of these new functions. 
To avoid the confusion, 
we split our statements into three lemmas. 
\par
Firstly, 
the properties $F_\alpha$ for $f\in\mathscr{S}_{c,m}(\mathbb{R}^2)$ 
are the following.  
\begin{lemma}
\label{theorem:change1}
For $f\in\mathscr{S}_{c,m}(\mathbb{R}^2)$, 
$F_\alpha(\xi,\eta)$ satisfies 
\begin{equation}
f(x,y)
=
\alpha\lvert{x-c}\rvert^{\alpha-1}
F_\alpha\bigl((x-c)\lvert{x-c}\rvert^{\alpha-1},y\bigr), 
\label{equation:inverse1} 
\end{equation}
and for any $N>0$, there exists a constant $C_N>0$ such that 
\begin{equation}
\lvert{F_\alpha(\xi,\eta)}\rvert
\leqq
C_N(1+\lvert\xi\rvert+\lvert\eta\rvert)^{-N}.
\label{equation:decay1} 
\end{equation}
Moreover, when $\pm\xi\downarrow0$, 
\begin{equation}
F_\alpha(\xi,s\xi+u)
\rightarrow
\begin{cases}
0 &\quad (m>\alpha-2),
\\
\dfrac{(\pm1)^{m+1}}{(m+2)!}
\cdot
\dfrac{\partial^{m+1} f}{\partial x^{m+1}}(c,u)
&\quad (m=\alpha-2). 
\end{cases} 
\label{equation:limit1}
\end{equation}
\end{lemma}
Here we remark that we interpret the right hand side of \eqref{equation:inverse1} at $x=c$ 
as zero since $\alpha>1$ and \eqref{equation:limit1}.
\begin{proof}[{Proof of Lemma~\ref{theorem:change1}}] 
Suppose that $f\in\mathscr{S}_{c,m}(\mathbb{R}^2)$. 
We can check \eqref{equation:inverse1} since 
$x\lvert{x}\rvert^{\alpha-1}=\xi$ 
and 
$\lvert{x}\rvert^\alpha=\lvert{\xi}\rvert$. 
Combining the definition of $F_\alpha$ and \eqref{equation:schwartz}, 
we obtain the decay estimate \eqref{equation:decay1} 
for $\lvert\xi\rvert>1$ and $\eta\in\mathbb{R}$. 
By using \eqref{equation:vanishing1}, we have for $\xi\ne0$ 
\begin{align*}
  F_\alpha(\xi,\eta)
& =
  \frac{f(\xi\lvert\xi\rvert^{-1+1/\alpha}+c,\eta)}{\alpha\lvert\xi\rvert^{(\alpha-1)/\alpha}}
\\
& =
  \frac{\bigl(\operatorname{sgn}(\xi)\bigr)^{m+1} \lvert\xi\rvert^{(m+2-\alpha)/\alpha}}{\alpha m!}
\\
& \times
  \int_0^1
  (1-t)^m
  \frac{\partial^{m+1} f}{\partial x^{m+1}}(t\xi\lvert\xi\rvert^{-1+1/\alpha}+c,\eta)
  dt. 
\end{align*}
By using this and \eqref{equation:schwartz}, 
we have $F_\alpha(\xi,\eta)=O\bigl((1+\lvert\eta\rvert)^{-N}\bigr)$ 
uniformly in $0<\lvert\xi\rvert\leqq1$ for any $N>0$ and $\eta\in\mathbb{R}$. 
Thus we proved \eqref{equation:decay1}. 
Using the above again, we have 
\begin{align*}
  F_\alpha(\xi,s\xi+u)
& =
  \frac{\bigl(\operatorname{sgn}(\xi)\bigr)^{m+1} \lvert\xi\rvert^{(m+2-\alpha)/\alpha}}{\alpha m!}
\\
& \times 
  \int_0^1
  (1-t)^m
  \frac{\partial^{m+1} f}{\partial x^{m+1}}(t\xi\lvert\xi\rvert^{-1+1/\alpha}+c,s\xi+u)
  dt. 
\end{align*}
This implies \eqref{equation:limit1} immediately.  
\end{proof}
Secondly, 
the properties $F_\alpha^\mathcal{P}$ for 
$f\in\mathscr{S}_{c,m}^\mathcal{P}(\mathbb{R}^2)$ 
are the following.  
\begin{lemma}
\label{theorem:change2}
For $f\in\mathscr{S}_{c,m}^\mathcal{P}(\mathbb{R}^2)$, 
$F_\alpha^\mathcal{P}(\xi,\eta)$ satisfies 
\begin{equation}
f(x,y)
=
\frac{\alpha}{2}
\lvert{x-c}\rvert^{\alpha-1}
F_\alpha^\mathcal{P}\bigl(\lvert{x-c}\rvert^\alpha,y\bigr), 
\label{equation:inverse2} 
\end{equation}
and for any $N>0$, there exists a constant $C_N>0$ such that 
\begin{equation}
\lvert{F_\alpha^\mathcal{P}(\xi,\eta)}\rvert
\leqq
C_N(1+\lvert\xi\rvert+\lvert\eta\rvert)^{-N}.
\label{equation:decay2} 
\end{equation}
Moreover, when $\xi\uparrow0$, 
$F_\alpha^\mathcal{P}(\xi,s\xi+u)\rightarrow0$, 
and when $\xi\downarrow0$, 
\begin{equation}
F_\alpha^\mathcal{P}(\xi,s\xi+u)
\rightarrow
\begin{cases}
0 &\quad (m>\alpha-2),
\\
\dfrac{2}{(m+2)!}
\cdot
\dfrac{\partial^{m+1} f}{\partial x^{m+1}}(c,u)
&\quad (m=\alpha-2). 
\end{cases} 
\label{equation:limit2}
\end{equation}
\end{lemma}
\begin{proof}
We can prove Lemma~\ref{theorem:change2} 
in the same way as the proof of Lemma~\ref{theorem:change1}. 
Here we omit the detail. 
\end{proof}
Finally, 
the properties $F_{\alpha,\beta}^\mathcal{R}$ for 
$f\in\mathscr{S}_{c,m}^\mathcal{R}(\mathbb{R}^2)$ 
are the following.  
\begin{lemma}
\label{theorem:change3}
For $f\in\mathscr{S}_{c,m}^\mathcal{R}(\mathbb{R}^2)$, 
$F_{\alpha,\beta}^\mathcal{R}(\xi,\eta)$ satisfies 
\begin{equation}
f(x,y)
=
\frac{\alpha}{2}
\lvert{x-c}\rvert^{\alpha-1}
\lvert{y}\rvert
F_{\alpha,\beta}^\mathcal{R}\bigl(\lvert{x-c}\rvert^\alpha,\lvert{y}\rvert^\beta\bigr), 
\label{equation:inverse3} 
\end{equation}
and for any $N>0$, there exists a constant $C_N>0$ such that 
\begin{equation}
\lvert{F_{\alpha,\beta}^\mathcal{R}(\xi,\eta)}\rvert
\leqq
C_N(1+\lvert\xi\rvert+\lvert\eta\rvert)^{-N}.
\label{equation:decay3} 
\end{equation}
Moreover, when $\xi\uparrow0$, 
$F_{\alpha,\beta}^\mathcal{R}(\xi,s\xi+u)\rightarrow0$,  
and when $\xi\downarrow0$, 
\begin{equation}
F_{\alpha,\beta}^\mathcal{R}(\xi,s\xi+u)
\rightarrow
\begin{cases}
0 
&\quad
(m>\alpha-2 \ \text{or} \ u\leqq0),
\\
\dfrac{1}{(m+2)!}
\displaystyle\int_0^1
\dfrac{\partial^{m+2} f}{\partial x^{m+1} \partial y}(c,\tau u^{1/\beta})
d\tau
&\quad
(m=\alpha-2 \ \text{and}\ u>0). 
\end{cases}
\label{equation:limit3} 
\end{equation}
\end{lemma} 
\begin{proof}
Suppose $f\in\mathscr{S}_{c,m}^\mathcal{R}(\mathbb{R}^2)$. 
By using the symmetries, we have 
$$
f(x,y)
=
f\bigl((x-c)+c,y\bigr)
=
f(\lvert{x-c}\rvert+c,\lvert{y}\rvert)
=
\frac{\alpha}{2}\lvert{x-c}\rvert^{\alpha-1}\lvert{y}\rvert
F_{\alpha,\beta}^\mathcal{R}(\lvert{x-c}\rvert^\alpha,\lvert{y}\rvert^\beta), 
$$
which proves \eqref{equation:inverse3}. 
To complete the proof of Lemma~\ref{theorem:change3}, 
it suffices to show 
\eqref{equation:decay3} for $\xi>0$ and $\eta>0$, 
and \eqref{equation:limit3} for $\xi\downarrow0$. 
\par
By using the definition of $F_{\alpha,\beta}^\mathcal{R}$ and \eqref{equation:schwartz}, 
we obtain the decay estimate \eqref{equation:decay3} for $\xi>1$ and $\eta>1$. 
So we show \eqref{equation:decay3} for $0<\xi\leqq1$ or $0<\eta\leqq1$. 
When $\xi>1$ and $0<\eta\leqq1$, 
\eqref{equation:vanishing2} with $l=0$ implies 
$$
F_{\alpha,\beta}^\mathcal{R}(\xi,\eta)
=
\frac{1}{\alpha\xi^{(\alpha-1)/\alpha}}
\int_0^1
\frac{\partial f}{\partial y}
(\xi^{1/\alpha}+c,\tau\eta^{1/\beta})
d\tau. 
$$
Combining this and \eqref{equation:schwartz}, we have  
$F_{\alpha,\beta}^\mathcal{R}(\xi,\eta)=O\bigl((1+\xi)^{-N}\bigr)$ 
for any $N>0$ and $\xi>0$ uniformly in $0<\eta\leqq1$. 
\par
When $0<\xi\leqq1$ and $\eta>1$, 
\eqref{equation:vanishing1} implies 
$$
F_{\alpha,\beta}^\mathcal{R}(\xi,\eta)
=
\frac{\xi^{(m+2-\alpha)/\alpha}}{\alpha m!\eta^{1/\beta}}
\int_0^1
(1-t)^m
\frac{\partial^{m+1} f}{\partial x^{m+1}}
(t\xi^{1/\alpha}+c,\eta^{1/\beta})
dt. 
$$
By using this and \eqref{equation:schwartz}, we obtain 
$F_{\alpha,\beta}^\mathcal{R}(\xi,\eta)=O\bigl((1+\eta)^{-N}\bigr)$ 
for any $N>0$ and $\eta>1$ uniformly in $0<\xi\leqq1$. 
\par
When $0<\xi\leqq1$ and $0<\eta\leqq1$, 
\eqref{equation:vanishing3} implies 
$$
F_{\alpha,\beta}^\mathcal{R}(\xi,\eta)
=
\frac{\xi^{(m+2-\alpha)/\alpha}}{\alpha m!}
\int_0^1
\int_0^1
(1-t)^m
\frac{\partial^{m+2} f}{\partial x^{m+1} \partial y}
(t\xi^{1/\alpha}+c,\tau\eta^{1/\beta})
dtd\tau
=
O(1). 
$$
Combining the above estimates, we obtain \eqref{equation:decay3}. 
\par
If $u\leqq0$, then $s\xi+u \rightarrow 0$ and 
$F_{\alpha,\beta}^\mathcal{R}(\xi,s\xi+u)\rightarrow0$ as $\xi\rightarrow+0$. 
Note that if $u>0$, then $s\xi+u>0$ near $\xi=0$. 
By using \eqref{equation:vanishing3} again, we have for $0<\xi\ll1$ 
\begin{align*}
  F_{\alpha,\beta}^\mathcal{R}(\xi,s\xi+u)
& =
  \frac{\xi^{(m+2-\alpha)/\alpha}}{\alpha m!}
\\
& \times
  \int_0^1
  \int_0^1
  (1-t)^m
  \frac{\partial^{m+2} f}{\partial x^{m+1} \partial y}
  \bigl(t\xi^{1/\alpha}+c,\tau(s\xi+u)^{1/\beta}\bigr) 
  dtd\tau.
\end{align*}
By using this, we obtain \eqref{equation:limit3} for $u>0$. 
This completes the proof. 
\end{proof}
%
%
%
%
\section{Proof of Main Theorems}
\label{section:proofs}
We begin with computing 
$\mathcal{Q}_{\alpha}f(s,u)$, 
$\mathcal{P}_{\alpha}f(s,u)$ 
and  
$\mathcal{R}_{\alpha,\beta}f(s,u)$.  
\begin{lemma}
\label{theorem:radontransform}
\quad
\begin{itemize}
\item[{\rm (i)}] 
For $f\in\mathscr{S}_{c,m}(\mathbb{R}^2)$, 
\begin{align}
  \mathcal{Q}_{\alpha}f(s,u)
& =
  \frac{1}{\sqrt{1+s^2}}
  \cdot
  \mathcal{X}F_\alpha
  \left(\operatorname{Arccot}(-s),\frac{u}{\sqrt{1+s^2}}\right),
\label{equation:radon1}
\\
  \partial_u\mathcal{Q}_{\alpha}f(s,u)
& =
  \frac{1}{1+s^2}
  \cdot
  (\partial_t\mathcal{X}F_\alpha)
  \left(\operatorname{Arccot}(-s),\frac{u}{\sqrt{1+s^2}}\right).
\label{equation:derivative1} 
\end{align}
\item[{\rm (ii)}] 
For $f\in\mathscr{S}_{c,m}^\mathcal{P}(\mathbb{R}^2)$, 
\begin{align}
  \mathcal{P}_{\alpha}f(s,u)
& =
  \frac{1}{\sqrt{1+s^2}}
  \cdot
  \mathcal{X}F_\alpha^\mathcal{P}
  \left(\operatorname{Arccot}(-s),\frac{u}{\sqrt{1+s^2}}\right),
\label{equation:radon2}
\\
  \partial_u\mathcal{P}_{\alpha}f(s,u)
& =
  \frac{1}{1+s^2}
  \cdot
  (\partial_t\mathcal{X}F_\alpha^\mathcal{P})
  \left(\operatorname{Arccot}(-s),\frac{u}{\sqrt{1+s^2}}\right).
\label{equation:derivative2} 
\end{align}
\item[{\rm (iii)}] 
For $f\in\mathscr{S}_{c,m}^\mathcal{R}(\mathbb{R}^2)$, 
\begin{align}
  \mathcal{R}_{\alpha,\beta}f(s,u)
& =
  \frac{1}{\sqrt{1+s^2}}
  \cdot
  \mathcal{X}F_{\alpha,\beta}^\mathcal{R}
  \left(\operatorname{Arccot}(-s),\frac{u}{\sqrt{1+s^2}}\right),
\label{equation:radon3}
\\
  \partial_u\mathcal{R}_{\alpha,\beta}f(s,u)
& =
  \frac{1}{1+s^2}
  \cdot
  (\partial_t\mathcal{X}F_{\alpha,\beta}^\mathcal{R})
  \left(\operatorname{Arccot}(-s),\frac{u}{\sqrt{1+s^2}}\right).
\label{equation:derivative3} 
\end{align}
\end{itemize}
\end{lemma}
\begin{proof} 
Firstly, we show (i). 
Suppose that $f\in\mathscr{S}_{c,m}(\mathbb{R}^2)$. 
Recall the definition of $\mathcal{Q}_{\alpha}f(s,u)$:
$$
\mathcal{Q}_{\alpha}f(s,u)
=
\int_{-\infty}^\infty
f(x+c,sx\lvert{x}\rvert^{\alpha-1}+u)
dx. 
$$
We want to make use of the change of variable 
$x=\xi\lvert\xi\rvert^{-1+1/\alpha}$ for $\xi\ne0$. 
For this reason, 
we regard the above integration as the sum of improper integrals 
on intervals $(-\infty,0)$ and $(0,\infty)$. 
Since $dx/d\xi=1/(\alpha\lvert\xi\rvert^{(\alpha-1)/\alpha})$ 
and $x\lvert{x}\rvert^{\alpha-1}=\xi$ for $\xi\ne0$, 
we have 
\begin{equation}
\mathcal{Q}_{\alpha}f(s,u)
=
\int_{-\infty}^\infty
\frac{f(\xi\lvert\xi\rvert^{-1+1/\alpha}+c,s\xi+u)}{\alpha\lvert\xi\rvert^{(\alpha-1)/\alpha}}
d\xi
=
\int_{-\infty}^\infty
F_\alpha(\xi,s\xi+u)
d\xi. 
\label{equation:change11}
\end{equation}
\eqref{equation:decay1} implies 
$F_\alpha(\xi,s\xi+u)=O\bigl((1+\lvert\xi\rvert)^{-2}\bigr)$, 
and \eqref{equation:limit1} shows that 
$F_\alpha(\xi,s\xi+u)$ has finite limits when $\pm\xi\downarrow0$.  
Then the computation in the above integration can be justified. 
\par
Now we apply another change of variable 
$\mathbb{R}\ni\xi \mapsto \sigma\in\mathbb{R}$ 
defined by 
$$
\xi=-\frac{\sigma}{\sqrt{1+s^2}}-\frac{su}{1+s^2}. 
$$
We can see 
$$
-\frac{s}{\sqrt{1+s^2}}=\cos\theta, 
\quad
\frac{1}{\sqrt{1+s^2}}=\sin\theta
$$
for $s\in\mathbb{R}$. 
When $s$ moves from $-\infty$ to $\infty$, 
$\theta$ moves from $0$ to $\pi$. 
Hence $s=-\cot\theta$, that is, $\theta=\operatorname{Arccot}(-s)$. 
In this case, 
$$
(\xi,s\xi+u)
=
\left(
\frac{u}{\sqrt{1+s^2}}
\cos\theta
-
\sigma\sin\theta,
\frac{u}{\sqrt{1+s^2}}
\sin\theta
+
\sigma\cos\theta
\right),
$$
and $d\xi/d\sigma=-1/\sqrt{1+s^2}$. 
Hence \eqref{equation:change11} becomes 
\begin{align*}
  \mathcal{Q}_{\alpha}f(s,u)
& =
  \frac{1}{\sqrt{1+s^2}}
  \int_{-\infty}^\infty
  F_\alpha
  \left(
  \frac{u}{\sqrt{1+s^2}}
  \cos\theta
  -
  \sigma\sin\theta,
  \frac{u}{\sqrt{1+s^2}}
  \sin\theta
  +
  \sigma\cos\theta
  \right)
  d\sigma
\\
& =
  \frac{1}{\sqrt{1+s^2}}
  \cdot
  \mathcal{X}F_\alpha
  \left(\operatorname{Arccot}(-s),\frac{u}{\sqrt{1+s^2}}\right), 
\end{align*}
which is \eqref{equation:radon1}. 
Differentiating this with respect to u, 
we can get \eqref{equation:derivative1}. 
\par
Secondly, we prove (ii). 
Differentiating \eqref{equation:radon2} with respect to u, 
one can get \eqref{equation:derivative2}. 
Here we show only \eqref{equation:radon2}. 
Suppose that $f\in\mathscr{S}_{c,m}^\mathcal{P}(\mathbb{R}^2)$. 
By using the symmetry $f(x+c,y)=f(-x+c,y)$ 
and the same change of variable 
$x \mapsto \xi$ in the proof of (i), we deduce that 
\begin{align*}
  \mathcal{P}_{\alpha}f(s,u)
& =
  \int_{-\infty}^\infty
  f(x+c,s\lvert{x}\rvert^\alpha+u)
  dx
\\
& =
  2
  \int_0^\infty
  f(x+c,s\lvert{x}\rvert^\alpha+u)
  dx   
\\
& =
  \int_0^\infty
  \frac{2f(\xi^{1/\alpha}+c,s\xi+u)}{\alpha\xi^{(\alpha-1)/\alpha}}
  d\xi
\\
& =
  \int_0^\infty
  F_\alpha^\mathcal{P}
  (\xi,s\xi+u)
  d\xi
\\
& =
  \int_{-\infty}^\infty
  F_\alpha^\mathcal{P}
  (\xi,s\xi+u)
  d\xi.
\end{align*}
\eqref{equation:decay2} implies 
$F_\alpha^\mathcal{P}(\xi,s\xi+u)=O\bigl((1+\lvert\xi\rvert)^{-2}\bigr)$, 
and \eqref{equation:limit2} shows that 
$F_\alpha^\mathcal{P}(\xi,s\xi+u)$ has finite limits when $\pm\xi\downarrow0$.  
Then the computation in the above integration can be justified. 
Using the same change of variable $\xi \mapsto \sigma$ 
as in the proof of \eqref{equation:radon1}, 
we obtain \eqref{equation:radon2}. 
\par
Finally, we prove (iii). 
Differentiating \eqref{equation:radon3} with respect to u, 
one can get \eqref{equation:derivative3}. 
Here we show only \eqref{equation:radon3}. 
Suppose that $f\in\mathscr{S}_{c,m}^\mathcal{R}(\mathbb{R}^2)$. 
By using the symmetry $f(x+c,y)=f(-x+c,y)$ 
and the same change of variable 
$x \mapsto \xi$ in the proof of (i), we deduce that 
\begin{align*}
  \mathcal{R}_{\alpha,\beta}f(s,u)
& =
  \int_{\substack{x\in\mathbb{R} \\ s\lvert{x}\rvert^\alpha+u>0}}
  \frac{f(x+c,\{s\lvert{x}\rvert^\alpha+u\}^{1/\beta})}{\{s\lvert{x}\rvert^\alpha+u\}^{1/\beta}}
  dx
\\
& =
  2
  \int_{\substack{x>0 \\ s\lvert{x}\rvert^\alpha+u>0}}
  \frac{f(x+c,\{sx^\alpha+u\}^{1/\beta})}{\{sx^\alpha+u\}^{1/\beta}}
  dx   
\\
& =
  2
  \int_{\substack{\xi>0 \\ s\xi+u>0}}
  \frac{f(\xi^{1/\alpha}+c,\{s\xi+u\}^{1/\beta})}{\alpha\xi^{(\alpha-1)/\alpha}\{s\xi+u\}^{1/\beta}}
  d\xi
\\
& =
  \int_0^\infty
  F_{\alpha,\beta}^\mathcal{R}
  (\xi,s\xi+u)
  d\xi
\\
& =
  \int_{-\infty}^\infty
  F_{\alpha,\beta}^\mathcal{R}
  (\xi,s\xi+u)
  d\xi.
\end{align*}
\eqref{equation:decay3} implies 
$F_{\alpha,\beta}^\mathcal{R}(\xi,s\xi+u)=O\bigl((1+\lvert\xi\rvert)^{-2}\bigr)$, 
and \eqref{equation:limit3} shows that 
$F_{\alpha,\beta}^\mathcal{R}(\xi,s\xi+u)$ has finite limits when $\pm\xi\downarrow0$.  
Then the computation in the above integration can be justified.Using the same change of variable $\xi \mapsto \sigma$ 
as in the proof of \eqref{equation:radon1}, 
we obtain \eqref{equation:radon3}. 
\end{proof}
Now we shall prove Theorem~\ref{theorem:main}. 
\begin{proof}[Proof of Theorem~\ref{theorem:main}]
Firstly, we prove \eqref{equation:inversion1}. 
Suppose that $f\in\mathscr{S}_{c,m}(\mathbb{R}^2)$. 
By using 
\eqref{equation:inverse1} in Lemma~\ref{theorem:change1} 
and Theorem~\ref{theorem:inversionX}, 
we deduce that 
\begin{align*}
  f(x,y)
& =
  \alpha\lvert{x-c}\rvert^{\alpha-1}
  F_\alpha\bigl((x-c)\lvert{x-c}\rvert^{\alpha-1},y\bigr)
\\
& =
  \frac{\alpha\lvert{x-c}\rvert^{\alpha-1}}{2\pi^2}
  \int_0^\pi
  \left(
  \operatorname{vp}
  \int_{-\infty}^\infty
  \frac{(\partial_t \mathcal{X}F_\alpha)(\theta,t)}{(x-c)\lvert{x-c}\rvert^{\alpha-1}\cos\theta+y\sin\theta-t}
  dt
  \right)
  d\theta.
\end{align*}
Here we use the change of variables $(\theta,t)=(\operatorname{Arccot}(-s),u/\sqrt{1+s^2})$, 
whose Jacobian is 
$$
\frac{\partial(\theta,t)}{\partial(s,u)}
=
\det
\begin{bmatrix}
1/(1+s^2) & 0
\\
\ast & 1/\sqrt{1+s^2} 
\end{bmatrix}
=
\frac{1}{(1+s^2)^{3/2}}. 
$$
Using this and 
\eqref{equation:derivative1} in Lemma~\ref{theorem:radontransform} in order, 
we deduce that 
\begin{align*}
  f(x,y)
& =
  \frac{\alpha\lvert{x-c}\rvert^{\alpha-1}}{2\pi^2}
  \int_{-\infty}^\infty
  \left(
  \operatorname{vp}
  \int_{-\infty}^\infty
  \right.
\\
& \qquad\qquad\times
  \cfrac{1}{-(x-c)\lvert{x-c}\rvert^{\alpha-1}\cfrac{s}{\sqrt{1+s^2}}+y\cfrac{1}{\sqrt{1+s^2}}-\cfrac{u}{\sqrt{1+s^2}}}
\\
& \qquad\qquad\times
  \left.
  \frac{1}{(1+s^2)^{3/2}}
  \cdot
  (\partial_t \mathcal{X}F_\alpha)\left(\operatorname{Arccot}(-s),\frac{u}{\sqrt{1+s^2}}\right)
  du
  \right)
  ds
\\
& =
  \frac{\alpha\lvert{x-c}\rvert^{\alpha-1}}{2\pi^2}
  \int_{-\infty}^\infty
  \left(
  \operatorname{vp}
  \int_{-\infty}^\infty
  \frac{1}{y-s(x-c)\lvert{x-c}\rvert^{\alpha-1}-u}
  \right.
\\
& \qquad\qquad\times
  \left.
  \frac{1}{1+s^2}
  \cdot
  (\partial_t \mathcal{X}F_\alpha)\left(\operatorname{Arccot}(-s),\frac{u}{\sqrt{1+s^2}}\right)
  du
  \right)
  ds
\\
& =
  \frac{\alpha\lvert{x-c}\rvert^{\alpha-1}}{2\pi^2}
  \int_{-\infty}^\infty
  \left(
  \operatorname{vp}
  \int_{-\infty}^\infty
  \frac{\partial_u \mathcal{Q}_{\alpha}f(s,u)}{y-s(x-c)\lvert{x-c}\rvert^{\alpha-1}-u}
  du
  \right)
  ds, 
\end{align*}
which is \eqref{equation:inversion1}. 
\par
Secondly, we prove \eqref{equation:inversion2}. 
Suppose that $f\in\mathscr{S}_{c,m}^\mathcal{P}(\mathbb{R}^2)$. 
By using 
\eqref{equation:inverse2} in Lemma~\ref{theorem:change2} 
and Theorem~\ref{theorem:inversionX}, 
we deduce that 
\begin{align*}
  f(x,y)
& =
  \frac{\alpha}{2}\lvert{x-c}\rvert^{\alpha-1}
  F_\alpha^\mathcal{P}(\lvert{x-c}\rvert^\alpha,y)
\\
& =
  \frac{\alpha\lvert{x-c}\rvert^{\alpha-1}}{4\pi^2}
  \int_0^\pi
  \left(
  \operatorname{vp}
  \int_{-\infty}^\infty
  \frac{(\partial_t \mathcal{X}F_\alpha^\mathcal{P})(\theta,t)}{\lvert{x-c}\rvert^\alpha\cos\theta+y\sin\theta-t}
  dt
  \right)
  d\theta.
\end{align*}
By using the change of variables $(\theta,t)=(\operatorname{Arccot}(-s),u/\sqrt{1+s^2})$ and 
\eqref{equation:derivative2} in Lemma~\ref{theorem:radontransform} in order, 
we can obtain \eqref{equation:inversion2} in the same way as \eqref{equation:inversion1}. 
Here we omit the detail. 
\par
Finally, we prove \eqref{equation:inversion3}. 
Suppose that $f\in\mathscr{S}_{c,m}^\mathcal{R}(\mathbb{R}^2)$. 
By using 
\eqref{equation:inverse3} in Lemma~\ref{theorem:change3} 
and Theorem~\ref{theorem:inversionX}, 
we deduce that 
\begin{align*}
  f(x,y)
& =
  \frac{\alpha}{2}\lvert{x-c}\rvert^{\alpha-1}\lvert{y}\rvert
  F_{\alpha,\beta}^\mathcal{R}(\lvert{x-c}\rvert^\alpha,\lvert{y}\rvert^\beta)
\\
& =
  \frac{\alpha\lvert{x-c}\rvert^{\alpha-1}\lvert{y}\rvert}{4\pi^2}
  \int_0^\pi
  \left(
  \operatorname{vp}
  \int_{-\infty}^\infty
  \frac{(\partial_t \mathcal{X}F_{\alpha,\beta}^\mathcal{R})(\theta,t)}{\lvert{x-c}\rvert^\alpha\cos\theta+\lvert{y}\rvert^\beta\sin\theta-t}
  dt
  \right)
  d\theta.
\end{align*}
By using the change of variables $(\theta,t)=(\operatorname{Arccot}(-s),u/\sqrt{1+s^2})$ and 
\eqref{equation:derivative3} in Lemma~\ref{theorem:radontransform} in order, 
we can obtain \eqref{equation:inversion3} in the same way as \eqref{equation:inversion1}. 
Here we omit the detail. 
This completes the proof.
\end{proof}
%
%
\begin{flushleft}
{\bf Acknowledgments} 
\end{flushleft}
The author would like to thank referees for reading the first version of the manuscript carefully and giving valuable comments.  
%
%
\begin{flushleft}
{\bf Disclosure statement} 
\end{flushleft}
No potential conflict of interest was reported by the author. 
%
%
\begin{flushleft}
{\bf Funding} 
\end{flushleft}
The author was supported by 
the JSPS Grant-in-Aid for Scientific Research \#19K03569.
%
%

\end{document}